\documentclass[12pt]{amsart}
\usepackage{amsmath}
\usepackage{amssymb}
\usepackage{amsfonts}
\usepackage{textcomp}
\usepackage{amsthm}
\usepackage{mathrsfs}
\usepackage[latin1]{inputenc}
\usepackage[all]{xy}
\usepackage{hyperref}
\usepackage{url}
\usepackage[paperheight=11in,paperwidth=8.5in,left=1in,top=1.25in,right=1in,bottom=1.25in,]{geometry}

\newcommand{\I}{\mathrm{I}}
\newcommand\II{\mathrm{I}\!\mathrm{I}}
\newcommand\III{\mathrm{I}\!\mathrm{I}\!\mathrm{I}}
\newcommand{\M}{\overline{M}}

\newtheorem{theorem}{Theorem}
\newtheorem{lem}{Lemma}
\newtheorem{prop}{Proposition}
\newtheorem{defi}{Definition}
\newtheorem{cor}{Corollary}

\begin{document}
\title[Local convexity of renormalized volume for cusped manifolds]{\textbf{Local convexity of renormalized volume for rank-1 cusped manifolds}}
\author{Franco Vargas Pallete}
\thanks{Research partially supported by NSF grant DMS-1406301}
\address{Department of Mathematics  \\
 University of California at Berkeley \\
745 Evans Hall \\
Berkeley, CA 94720-3860 \\
U.S.A.}
\email{franco@math.berkeley.edu}

\begin{abstract}
We study the critical points of the renormalized volume for acylindrical geometrically finite hyperbolic 3-manifolds that include rank-1 cusps, and show that the renormalized volume is locally convex around these critical points. We give a modified definition of the renormalized volume that is additive under gluing, and study some local properties.
\end{abstract}
\maketitle

\section{Introduction}
Renormalized volume is a quantity that gives a notion of volume for hyperbolic manifolds which have infinite volume under the classical definition. Its study for convex co-compact hyperbolic 3-manifolds can be found in \cite{KrasnovSchlenker}, while the geometrically finite case which includes rank 1-cusps has been developed in \cite{MoroianuGuillarmouRochon}.

The article is organized as follows: sections (\ref{sec:Localparam}) and (\ref{sec:review}) give a review of parametrization of hyperbolic metrics and renormalized volume for convex co-compact manifolds. Section (\ref{sec:onmaximality}) gives a way to complete a result of \cite{KrasnovSchlenker} using Ricci flow. Section (\ref{sec:localconvexity}) gives a proof for the convex co-compact case by showing that the Hessian of the renormalized volume is positive definite, while section (\ref{sec:remarks}) shows how to extend this proof when rank-1 cusps are taken into account. Finally, in section (\ref{sec:corrected}) we define the corrected renormalized volume, which is additive under gluing, and we prove convexity at one of its critical points.

Moroianu \cite{Moroianu} has proved independently that the Hessian is positive definite at the critical points for the convex co-compact case, by the use of minimal surfaces. Our method relies on computing the Hessian for quasi-Fuchsian manifolds, and then using the skinning map to compute it for acylindrical manifolds.

\textbf{Acknowledgements:} I would like to thank Ian Agol for introducing me to the topic and for his guidance during this project. The definition of corrected renormalized volume is due to him. I would also like to thank David Dumas, Richard Canary and Martin Bridgeman for their helpful comments, as well to the anonymous referee.

\section{Local parametrization of hyperbolic metrics}\label{sec:Localparam}

Let $R$ be a Riemann surface of genus $g$ with $n$ punctures, $\mathcal{T}$ its Teichmuller space, $\mathcal{B}$ its space of Beltrami differentials and $Q$ its space of quadratic holomorphic differentials. There is a nice description of these two last spaces in terms of complex-valued functions in the following way (it can be found, for example, in \cite{Gardiner}):

Take the covering map $\mathbb{H}^2 \rightarrow R$ and denote by $\Gamma\subseteq {\rm PSL}(2,\mathbb{R})$ the group of deck transformations. Then elements of $\mathcal{B}$ are represented in $\mathbb{H}^2$ by measurable $L^\infty$ functions $\mu$ such that
\begin{equation}\label{eq:beltrami_invariance}
	\mu(A(z)) \overline{A'(z)} = \mu(z) A'(z), \forall A\in\Gamma.
\end{equation}
Similarly, elements of $Q$ are represented in $\mathbb{H}^2$ by holomorphic functions $\phi$ such that
\begin{equation}
	\phi(A(z)) A'(z)^2 = \phi(z),\forall A\in\Gamma
\end{equation}
and $\iint_{R(\Gamma)} \vert\phi(z)\vert dx\wedge dy < \infty$, where $R(\Gamma)$ is any fundamental domain for $\Gamma$ in $\mathbb{H}^2$.
Observe that the hyperbolic metric on $\mathbb{H}^2$, $\frac{1}{y^2}dzd\overline{z}$, corresponds to $\rho(z)=\frac{1}{y^2}$ which has the property
\begin{equation}
	\rho(A(z)) A'(z)\overline{A'(z)} = \rho(z), \forall A\in {\rm PSL}(2,\mathbb{R}),
\end{equation}
and hence $\frac{\overline{\phi}}{2\rho}$ satisfies (\ref{eq:beltrami_invariance}). Because $\iint_{R(\Gamma)} \vert\phi(z)\vert dx\wedge dy < \infty$, $\frac{\overline{\phi}}{2\rho}$ is in $L^\infty$, and then we have a map $Q \rightarrow \mathcal{B}, \phi \mapsto \frac{\overline{\phi}}{2\rho}$.

Thanks to the theory of the Beltrami equation (\cite{Ahlfors}, \cite{AhlforsBers}, \cite{LehtoVirtanen}), we can find a solution of $f_{\overline{z}}(z)= \mu(z)f_z(z)$ for $\Vert\mu\Vert_\infty <1$ by a unique quasiconformal self-mapping $f^\mu$ of $\mathbb{H}^2$ which extends continuously to $\partial\mathbb{H}^2$, fixes $0,1,\infty$ and depends analytically on $\mu$ ($f$ solves the Beltrami equation in the distributional sense). With this we have a map from the unit ball of $\mathcal{B}$ to $\mathcal{T}$, and by Teichm\"uller theory, the correspondence that sends $\phi \mapsto\frac{\overline{\phi}}{2\rho}$ and then to the solution of the Beltrami equation for $\frac{\overline{\phi}}{2\rho}$ defines a local homeomorphism between a neighbourhood of $0$ in $Q$ to a neighbourhood of $\mathcal{T}$.

Now, the solution to the Beltrami equation for $\mu= \frac{\overline{\phi}}{2\rho}$, $(f^\mu)^*(\rho dzd\overline{z})$ is a local parametrization of hyperbolic metrics defined on the same space. We would like to compute the variation of this family of metrics at the origin. Because $f^{t\mu}$ satisfies the Beltrami equation for $t\mu$, we have
\begin{equation}
	(f^{t\mu})^*(\rho dzd\overline{z})(z) = \rho(f^{t\mu(z)})\vert f^{t\mu}_z\vert^2 \vert dz + t\mu(z)d\overline{z}\vert^2.
\end{equation}
Hence (here is implicit the analytic dependence of $f^\mu$ with respect to $\mu$)
\begin{equation}
	\frac{\partial}{\partial t}\bigg|_{t=0} (f^{t\mu})^*(\rho dzd\overline{z})(z) = \rho(z)\mu(z)d\overline{z}^2 + \rho\overline{\mu(z)}dz^2 + Edzd\overline{z},
\end{equation}
where $E$ groups the terms of the derivative that go together with $dzd\overline{z}$. We can then replace $\mu= \frac{\overline{\phi}}{2\rho}$ to obtain
\begin{equation}\label{eq:metricvariation}
	\frac{\partial}{\partial t}\bigg|_{t=0} (f^{t\mu})^*(\rho dzd\overline{z})(z) =	{\rm Re}(\phi dz^2) + Edzd\overline{z}.
\end{equation}

Define $RQ$ as the space of real parts of quadratic holomorphic differentials. Since taking the real part defines an isomorphism $Q\rightarrow RQ$, we have a local homeomorphism from a neighbourhood of $0$ in $RQ$ to a neighbourhood of $\mathcal{T}$. Moreover, if we define $\I_v$ to be the hyperbolic metric $(f^\mu)^*(\rho dzd\overline{z})$, where $\mu= \frac{\overline{\phi}}{2\rho}$ and ${\rm Re}(\phi dz^2) = v$, then (\ref{eq:metricvariation}) implies that at $0$
\begin{equation}
	D\I_v = v + Edzd\overline{z}.
\end{equation}
Using the hyperbolic metric of $R$ to define an inner product for tensors, $D\I_v$ projects orthogonally to $v$ (because $dz^2, d\overline{z}^2$ are pointwise orthogonal to $dzd\overline{z}$). In particular, for $v,w\in RQ_c$ we have
\begin{equation}\label{eq:orthogonalprojection}
	\langle D\I_c(v), w\rangle = \langle v, w\rangle.
\end{equation}
Observe finally that if we had chosen $R$ with the same hyperbolic metric but with opposite orientation, the space $RQ$ coincides with the one defined for the original Riemann surface structure.

\section{Review of renormalized volume}\label{sec:review}

Given a convex co-compact hyperbolic 3-manifold $(M,g)$, Krasnov and Schlenker \cite{KrasnovSchlenker} defined its renormalized volume and calculated its first variation from the $W$-volume of a compact submanifold $N$ (\cite{KrasnovSchlenker} Definition 3.1) as
\begin{equation}\label{Wvol}
	W(M,N) = V(N) - \frac{1}{4}\int_{\partial N}Hda,
\end{equation}
where $da$ is the area form of the induced metric.

Expanding on the notation used in \cite{KrasnovSchlenker}, denote by $\I$ the metric induced on $\partial N$, $\II$ its second fundamental form (so $\II(x,y) = \I(x,By)$, where $B$ is the shape operator) and $\III(x,y) = \I(Bx,By)$ its third fundamental form.

If we further assume that $N$ has convex boundary and that the normal exponential map (pointing towards the exterior of $\partial N$) defines a family of equidistant surfaces $\lbrace S_r\rbrace$ that exhaust the complement of $N$ ($S_0= \partial N$), then the $W$-volume of $N_r$ (points on the interior of $S_r$) satisfies (\cite{Schlenker13} Lemma 3.6)
\begin{equation}\label{Wr}
	W(N_r) = W(N) - \pi r \chi(\partial N).
\end{equation}
Also, as observed in (\cite{Schlenker13} Definition 3.2, Proposition 3.3), $\I^* = 4\displaystyle{\lim_{r\rightarrow\infty}e^{-2r}\I_r}$ (where $\I_r$ is the metric induced on $S_r$, which is identified with $S$ by the normal exponential map) exists and lies in the conformal class of the boundary. The analogous re-scaled limits for $\II, \III, B$ also exist and are denoted by $\II^*, \III^*, B^*$. The reason to multiply by $4$ is so $\I^*= g|_N$ in the case when $N$ is a totally geodesic surface.

For the case of convex co-compact manifolds, any metric at infinity that belongs to the conformal class given by the hyperbolic structure can be obtained as the rescaled limit of the induced metrics of some family of equidistant surfaces. Theorem 5.8 of \cite{KrasnovSchlenker} describes this by the use of Epstein surfaces (as stablished in \cite{Epstein}), which in turn allows us to define
\begin{equation}
	W(M,h) = W(M,N_r) + \pi r \chi(\partial M),
\end{equation}
where $\lbrace N_r\rbrace$ corresponds to the equidistant surfaces given by the Epstein surfaces of $h$ (in other words, $h=\I^*$). Then $W(M,h)$ is well-defined as a consequence of \eqref{Wr}.

We can finally define the renormalized volume of $M$ as
\begin{equation}
	V_R(M) = W(M,h),
\end{equation}
where $h$ is the metric in the conformal class at infinity that has constant curvature $-1$.

Krasnov-Schlenker \cite{KrasnovSchlenker} derived the variation formula of the $W$-volume in terms of the input at infinity (observe that because of the description by Epstein surfaces, $\I^*$ determines $\II^*$ and $\III^*$) from the volume variation of Rivin-Schlenker \cite{RivinSchlenker}. Since for the rest of the paper we are going to write everything in terms of these limit tensors, let us omit the superscript $^*$.

As we did in the previous section, let us fix $c\in\mathcal{T}(\partial M)$ and some metric $\I_c$ that represents it, so we can parametrize $\mathcal{T}(\partial M)$ by $RQ_c$. Then for $v\in RQ_c$ we have the variation of $V_R$ at $I_c$ (\cite{KrasnovSchlenker} Corollary 6.2, Lemma 8.5)
\begin{equation}
	DV_R(v) = -\frac{1}{4}\int_{\partial M}\langle D\mathrm{I}_c(v),\II_0 \rangle da,
\end{equation}
where the metric between tensors and the area form $da$ are defined from $\I_c$ and $\II_0 = \II - \frac{1}{2}\I$ is the traceless second fundamental form. This 2-form is (at each component of $\partial M$, after taking quotient by the action of $\pi_1(M)$) the negative of the real part of the Schwarzian derivative of the holomorphic map between one component of the region of discontinuity and a disk (\cite{KrasnovSchlenker} Lemma 8.3). In particular (as we stated in (\ref{eq:orthogonalprojection})) $\langle D\I_c(v), \II_0\rangle = \langle v, \II_0\rangle$ pointwise. Then if we take $c$ to be a critical point (i.e.\ $DV_R(v)=0$ at $\I_c$ for every $v\in RQ_c$) $\II_0$ must vanish at every point. This in turn implies that the holomorphic map between a component of the region of discontinuity and a disk has Schwarzian derivative identically zero, which means that the components are disks and the boundary of the convex core is totally geodesic.

\section{On the maximality of the $W$-volume among metrics of constant area}\label{sec:onmaximality}

In Section 7 of \cite{KrasnovSchlenker}, Krasnov and Schlenker study the variation of the $W$-volume among metrics of the same conformal class while keeping the area constant. One way of showing this is by observing that the Ricci flow is a gradient-like flow for the $W$-Volume. We include this argument due to the connection to Ricci flow, since earlier proofs of this fact appeared in \cite{MoroianuGuillarmouRochon} (Proposition 7.1, which includes the cusped case) and previously in \cite{GuillarmouMoroianuSchlenker} (Proposition 3.11, convex co-compact case).

From \cite{KrasnovSchlenker} Corollary 6.2 we know that
\begin{equation}
	\delta W = -\frac{1}{4} \int_{\partial M} \delta H + \langle \delta \I, \II_0\rangle da,
\end{equation}
where we omit the superscript $^*$ and $H$ denotes the mean curvature at infinity.

Now, since we are taking a conformal variation, $\delta \I = u\I$ for some function $u: \partial M \rightarrow \mathbb{R}$. Moreover, since the volume is preserved, $\int_{\partial M} uda = 0$.

Remember that $\II_0$ is the traceless part of the second fundamental form, so $\langle u\I,\II_0\rangle =0 $ pointwise. Also (\cite{KrasnovSchlenker} Remark 5.4) $H=-K$ and hence we have
\begin{equation}
	\delta W = \frac{1}{4} \int_{\partial M} \delta K da.
\end{equation}
But by the Gauss-Bonnet formula $\int_{\partial M} \delta K da + \int_{\partial M}  K \delta(da) = 0$ and the equality $\delta(da) = \frac{1}{2}\langle \delta\I,\I\rangle da = uda$, we can reduce it to
\begin{equation}
	\delta W = -\frac{1}{4}\int_{\partial M}Kuda,
\end{equation}
from where we can recover that $K = const.$ is the unique critical point. If two points had different values of $K$, we can take $u$ supported around those points such that $\int_{\partial M} uda = 0$, but $-\int_{\partial M}Kuda > 0$.

Note that $u= -K + \frac{2\pi\chi(\partial M)}{\text{vol}(\partial M)}$ has integral equal to 0 ($\int_{\partial M}Kda = 2\pi\chi(\partial M)$), and by H\"{o}lder inequality
\begin{equation}
	\left(\int_{\partial M}K^2da\right). \left(\text{vol}(\partial M)\right) \geq \left(\int_{\partial M} Kda\right)^2,
\end{equation}
giving $-\frac{1}{4}\int_{\partial M} Kuda \geq 0$. Hence the $W$-volume is no decreasing under the Ricci flow in two dimensions. It is known (see, for example, \cite{IsenbergMazzeoSesum}) that this flow converges to the metric of constant curvature, proving that this metric is a global maximum. Since $K = const.$ is the only critical point, the $W$-volume increases strictly under the flow, making this metric a strict global maximum.

\section{Local convexity at the geodesic class}\label{sec:localconvexity}

In order to study local behavior around the critical points we want to compute the Hessian of $V_R$ at these points. Let $c$ be a critical point (i.e.\ $\II_0 \equiv 0$) and $\I_c$ a metric representing this class. For $v,w\in RQ_c$, we vary (6) with respect to $w$ to obtain
\begin{align}\label{eq:secondvaritation}\begin{aligned}
	\text{Hess}V_R(v,w) &= -\frac{1}{4}\int_{\partial M}\langle D\mathrm{I}(v),D\II_0(w) \rangle da 
	\\ &+ (\text{terms obtained by varying the other tensors }\text{ w.r.t. w}),
	\end{aligned}
\end{align}
and since $\II_0$ vanishes identically, all the terms in parenthesis get canceled, so
\begin{equation}\label{eq:hessian}
	\text{Hess}V_R(v,w) = -\frac{1}{4}\int_{\partial M}\langle D\mathrm{I}(v),D\II_0(w) \rangle da.
\end{equation}

Let us first understand the quasi-Fuchsian case (here we are referring to hyperbolic structures on the product $S\times \mathbb{R}$, where $S$ is a closed orientable surface of genus $g>1$). In order to differentiate the two ends let us label them (as well as the tensors defined on each one) with $+$ and $-$.

\begin{theorem} Let $M$ be a Fuchsian manifold (i.e.\ the conformal classes at infinty $c_+, c_-$ both equal to say a conformal class $c$). Then the Hessian at $M$ of the renormalized volume is positive definite in the orthogonal complement of the diagonal of $RQ_c \times RQ_c \approx T_c\mathcal{T}(\partial M) = T_{c^+}\mathcal{T}(\partial M^+)\times  T_{c^-}\mathcal{T}(\partial M^-)$  (where we are using the parametrization by real parts of holomorphic quadratic differentials w.r.t. $(c,c)$ and that the real parts of holomorphic quadratic differentials are the same for both orientations). Moreover, the Hessian agrees with the metric induced by $\mathrm{I}$ with a $\frac{1}{8}$ factor.

\end{theorem}

\begin{proof}
Recall by (\ref{eq:hessian}) that since $M$ is a critical point, then for $v=(v_+,v_-), w=(w_+,w_-)$ tangent vectors at $(c,c)$
\begin{equation}
	4\text{Hess} V_R (v,w)= -\int_{\partial_+ M}\langle D\I^+(v), D\II^+_0(w) \rangle da - \int_{\partial_- M}\langle D\I^-(v), D\II^-_0 (w)\rangle da.
\end{equation}

As mentioned in  \cite{KrasnovSchlenker}(Lemma 8.3) we know that $\II^+_0(c_+,c)$ is equal to $-{\rm Re}(q_+(c_+,c))$, where $q_+(c_+,.): \mathcal{T}_- \rightarrow Q_{c_+}$ is the Bers embedding. In particular $D\II^+_0(v,0)$ lands in $RQ_c$, and because $D\II^+_0(v,v)=0$, then the whole image of $D\II_0^+$ lands in $RQ_c$. Since an analogous argument works for $D\II^-_0$, we can then reduce to
\begin{align}\begin{aligned}
	4\text{Hess} V_R (v,w)&= -\int_{\partial_+ M}\langle v_+, D\II^+_0(w) \rangle da - \int_{\partial_- M}\langle v_-, D\II^-_0 (w)\rangle da \\&= -\langle v, D\II_0(w)\rangle_{L^2}, 
	\end{aligned}
\end{align}
where $\langle\cdot,\cdot\rangle$ on forms denotes the $L^2$ scalar product defined by $\I_c$.

Now $D\II_0$ is diagonalizable with orthogonal eigenvectors (is the expression of the Hessian in terms of the metric induced by $\I(c,c)$). We can exploit this fact in the following lemma.

\begin{lem}\label{II0var}
 $D\II_0(v,-v)=-\frac{1}{2}(v,-v)$
\end{lem}

\begin{proof}
 We prove first the following claim.
 
 \textbf{Claim:} Let $M$ be a Fuchsian manifold with associated conformal class c. Then $D\III^+(v,0)$ is orthogonal to $RQ_c$ for all $v\in RQ_c$.
 
Recall from \cite{KrasnovSchlenker} (as stated in the comments of Definition 5.3) that almost-Fuchsian manifolds are quasi-Fuchsian manifolds with the principal curvatures at infinity between $-1, 1$, and for those manifolds $\III$ is conformal to $\mathrm{I}^-$ , so $D\III^+(v,0)$ is $\mathrm{I}(c,c)$ multiplied by some function ($\III$ stays conformal to $\mathrm{I}(c,c)$ when we only varied $c_+$). To see that those tensors are orthogonal to $RQ_c$ (note that for every tensor space we are taking the inner product induced by $I_c$ and integration), recall that if we take conformal coordinates, elements of $RQ_c$ are expressed in terms of $dz^2$ and $d\overline{z}^2$, where the metric is in terms of $dzd\overline{z}$, meaning that $D\III^+(v,0)$ is even pointwise orthogonal to any element of $RQ_c$.

Going back to the proof of the lemma, it follows from \cite{KrasnovSchlenker} (Definition 5.3) that the first variations at $M$ satisfy
\begin{eqnarray}
	\delta\II_0 = \delta\II - \frac{1}{2}\delta\mathrm{I} =  \mathrm{I}(\delta B\cdot,\cdot)\\	
	\delta\III = \frac{1}{4}\delta\mathrm{I} + \mathrm{I}(\delta B\cdot,\cdot),
\end{eqnarray}
hence
\begin{equation}
	D\III = D\II_0 + \frac{1}{4}D\mathrm{I}.
\end{equation}

In particular, by the claim, $D\II_0(v,0) = -\frac{1}{4}v$, and so
\begin{equation}
	D\II^+_0(v,-v) = D\II^+_0(2v,0) + D\II^+_0(-v,-v) = -\frac{1}{4}D\mathrm{I^+}(2v,0) = -\frac{1}{2}v.
\end{equation}
The lemma follows since $D\II_0$ is diagonalizable in the orthogonal complement of the diagonal (the diagonal is part of the $0$-eigenspace).
\end{proof}
This lemma implies that the orthogonal complement is the $-\frac{1}{2}$-eigenspace for $D\II_0$, concluding the last part of the theorem.
\end{proof}

Now we will use this local behavior for quasi-Fuchsian manifolds to conclude our main result for acylindrical manifolds.

\begin{theorem}
Let $M$ be a compact acylindrical 3-manifold with hyperbolic interior such that $\partial M \neq \emptyset$. Then there is a unique critical point $c$ for the renormalized volume of $M$, where $c$ is the unique conformal class at the boundary that makes every component of the region of discontinuity a disk(a.k.a.\ the geodesic class). The Hessian at this critical point is positive definite.
\end{theorem}

\begin{proof}
Since $DV_R(v)= -\frac{1}{4}\langle D\mathrm{I}_c(v),\II_0 \rangle =  -\frac{1}{4}\langle v, \II_0 \rangle$ we have
\begin{equation}
	DV_R=0 \Leftrightarrow \II_0\equiv 0.
\end{equation}
Now, $\II_0\equiv 0$ at each boundary component corresponds to the nullity of the Schwarzian derivative between every component of the domain of discontinuity and a disk. This implies as we stated that every component of the region of discontinuity is a disk, and since our manifold is acylindrical, there is a unique such point in the Teichm\"{u}ller space of the boundary. We can also observe that for this point the boundary of the convex core is completely geodesic.

To prove that it is a strict local minimum, it is sufficient to prove that the Hessian is positive definite at this point. Let $\partial M = S_1 \cup \ldots \cup S_n$, and $c=(c_1,\ldots,c_n)$ denote the geodesic class. To show that the Hessian is positive definite, take parametrization $RQ$ for $S_i$ based at $c_i$ and use the same metric to compute variations of $V_R$ for $M$ and for both ends of quasi-Fuchsian manifolds $S_i\times\mathbb{R}$.

Recall that, since $M$ is hyperbolic acylindrical, the subgroups associated to the components of $\partial M$ are quasifuchsian. Hence we have a map from $\mathcal{T}(\partial M)$ (corresponding to hyperbolic metrics in $M$) to $\mathcal{T}(\partial M) \times \mathcal{T}(\overline{\partial M})$ (corresponding to the quasifuchsian subgroups). The first coordinate of this map is the identity on $\mathcal{T}(\partial M)$, while the second coordinate $\sigma:\mathcal{T}(\partial M) \rightarrow \mathcal{T}(\overline{\partial M})$ is Thurston's \textit{skinning map}.  Observe then that $\mathrm{I}$ at $S_i$ coincides with $\mathrm{I}^+(c_i,\sigma_i(c))$,  where $\sigma_i$ is the image of the skinning map $\sigma$ corresponding to $S_i$.
From the dependence  of $\II$ in terms of $\mathrm{I}$ (Remark 5.9 of \cite{KrasnovSchlenker}) we see that $\II_0$ at $S_i$  coincides with $\II^+_0(c_i,\sigma_i(c))$, and hence $D\I(v), D\II_0(w)$ are equal to $D\I^+(v_i,d\sigma_i(v)), D\II^+_0(w_i,d\sigma_i(w))$ at $S_i$ (note that here $d\sigma$ is written in terms of our charts given by $RQ_c$, so it is essentially the conjugation of the derivative of the skinning map, given our remark on how our charts behave with an orientation change). Hence
\begin{equation}\label{eq:hessianacyl}
	4\text{Hess}V_R(v,w) = -\sum^n_{i=1} \langle v_i, D\II^+_0(w_i,d\sigma_i(w))\rangle 
	= \frac{1}{4}\sum^n_{i=1} \langle v_i, w_i - d\sigma_i(w)  \rangle,
\end{equation}
which is greater than zero for $v\neq 0$ according to the following result:

\begin{theorem} [McMullen \cite{McMullen90}]
Under the conditions above, $\Vert d\sigma\Vert < 1$, where the norm $\Vert d\sigma \Vert$ is calculated in terms of the Teichm\"uller metric.
\end{theorem}

Indeed, (\ref{eq:hessianacyl}) tells us that $d\sigma_c$ is diagonalizable, which together with McMullen's theorem implies that all eigenvalues are less than 1. Then $\text{Hess}V_R$ is also diagonalizable with all eigenvalues positive.

\end{proof}

McMullen's result can be sketched as follows. The skinning map is holomorphic between Teichm\"{u}ller spaces, in particular sending holomorphic disks to holomorphic disks. This implies that $\Vert d\sigma\Vert \leq 1$ with respect to the Kobayashi metric, which coincides with the Teichm\"{u}ller metric. If $\Vert d\sigma\Vert = 1$ at a point, then $\sigma$ would be an isometry on the extremal holomorphic disk, but an earlier result of Thurston states that $\sigma$ is a strict contraction for the acylindrical case.

\begin{cor} Let $c\in \mathcal{T}(\partial M)$ be as in the previous theorem. Then $d\sigma$ admits an orthonormal eigenbasis with respect to the $L^2$-norm on $RQ_c$, with all eigenvalues less than $1$ in absolute value.
\end{cor}
Observe that if we take holomorphic quadratic differentials for the tangent space, we have the conclusion for $d\sigma$ after taking a complex conjugation.

This result was somehow expected thanks to a parallel between the skinning map and the Thurston map for postcritically finite rational maps, since the Thurston map is contracting for non-Latt\`e maps. Moreover, if $f$ is a postcritically finite quadratic polynomial, there is a uniform spectral gap for the derivative at its unique critical point (more precisely, all eigenvalues are greater than $1/8$ in norm. For this refer to \cite{BuffEpsteinKoch}). It is unknown to the author if there is an analog for hyperbolic 3-manifolds, namely a family of manifolds for which there is some uniform espectral gap.

It is worth mentioning  that it is still open that the geodesic class is the absolute minimum of $V_R$ (for acylindrical manifolds). It is also open that $V_R$ is strictly positive for quasi-Fuchsian manifolds outside the Fuchsian locus, although this is known for almost-Fuchsian manifolds \cite{CiobotaruMoroianu}.

\section{Remarks for hyperbolic 3-manifolds with rank-1 cusps}\label{sec:remarks}

Consider a quasi-Fuchsian manifold $M$ with rank-1 cusps and a compact subset $K$ of $\partial M^+$ (the top boundary component as we denoted in (\ref{sec:localconvexity})). For a quasi-Fuchsian manifold sufficiently close to the Fuchsian locus, the equidistant foliation over $K$ extends up to the 0-leaf and has principal curvatures between $-1$ and $1$. As stated in the comments of \cite{KrasnovSchlenker} Definition 5.3, this foliation extends to the other end of $M$ and $\III^+$ is a metric conformal to $I^-$ (all of this over $K$). Hence, at a Fuchsian manifold, $D\III^+(v,0)$ is a multiple of $I^+$ at any such compact $K$, and hence over all of $\partial M^+$.

On the other hand, all the formulas of \cite{KrasnovSchlenker} Theorem 5.8 apply for the rank 1-cusps case, so in particular we also have that $\II_0$ is (at each component of $\partial M$ after taking quotient by the action of $\pi_1(M)$) the negative of the real part of the Schwarzian derivative of the holomorphic map between one component of the region of discontinuity and a disk.

As we mentioned, the variation formula of the renormalized volume was proved in \cite{MoroianuGuillarmouRochon}. The second variation at a critical point is well defined since for the quasi-Fuchsian case it is the inner product of quadratic holomorphic differentials, which are in $L^2$ with respect to the metric. The other terms that get canceled with $\II_0$ in (\ref{eq:secondvaritation}) do not overcome the exponential decay of the metric since each of them have at most polynomial growth. The general case will then be well defined since it is also an inner product of quadratic holomorphic differentials, thanks to the skinning map argument.

Finally, McMullen's result still applies for surfaces with punctures, so the proof of our theorem extends to the case of geometrically finite hyperbolic acylindrical manifolds with rank 1 cusps. We summarize this as follows.

\begin{theorem} Let $M$ be a Fuchsian manifold (i.e.\ the conformal classes at infinty $c_+, c_-$ coincide, say to a conformal class $c$) with possibly rank 1-cusps.Then the Hessian at $M$ of the renormalized volume is positive definite in the orthogonal complement of the diagonal of $RQ_c \times RQ_c \approx T_c\mathcal{T}(\partial M) = T_{c^+}\mathcal{T}(\partial M^+)\times  T_{c^-}\mathcal{T}(\partial M^-)$  (where we are using the parametrization by real parts of holomorphic quadratic differentials w.r.t. $(c,c)$ and that the real parts of holomorphic quadratic differentials are the same for both orientations). Moreover, the Hessian agrees with the metric induced by $\mathrm{I}$ with a $\frac{1}{8}$ factor.

\end{theorem}

\begin{theorem}
Let $M$ be a geometrically finite manifold with rank 1-cusps. Then there is a unique critical point $c$ for the renormalized volume of $M$, where $c$ is the unique conformal class at the boundary that makes every component of the region of discontinuity a disk(a.k.a.\ the geodesic class). The Hessian at this critical point is positive definite.
\end{theorem}

We can also extend the Ricci flow argument for the maximality of the $W$-volume. We use the existence and convergence properties of the Ricci flow from the conditions stated in \cite{JiMazzeoSesum}. In the language of \cite{MoroianuGuillarmouRochon}, let a cusp be parametrized as $(v,w)\in \left[ 0,\frac{1}{R}\right[ \times\mathbb{R}\slash \frac{1}{2}\mathbb{Z}$, which under the notation of \cite{JiMazzeoSesum} corresponds to $(s,w)$,  where $v= e^{-s}$.

We need to show that $\psi\in C^{\infty}_r(\overline{\partial M}) \Rightarrow e^{2\phi}\in R + s^{-\mu}C^{2,\alpha}$ (where the first term describes the family of conformal factors for which the renormalized volume is defined in \cite{MoroianuGuillarmouRochon}, and the second term the sufficient conditions for running Ricci flow in the lines of \cite{JiMazzeoSesum}) for some constants $\mu>0, R;$ where $v= e^{-s}$, $\psi(v,w)= \phi(s,w)$. $C^{\infty}_r(\overline{\partial M})$ denotes (\cite{MoroianuGuillarmouRochon} 2.7) functions whose $w$-derivative vanishes with infinite order at $v=0$. $C^{2,\alpha}$ denotes the usual H\"{o}lder space with respect to the hyperbolic metric on the cusp.

Since $\psi\in C^{\infty}_r(\M)$, $\lim_{v\rightarrow 0}\psi(v,w)$ exists and does not depend on $w$, which we denote by $R$. Then we need to show that $a(s,w)= e^{2\phi(s,w)}-R$ belongs to $s^{-\mu}C^{2,\alpha}$. We will show that belongs to $s^{-\mu}C^3$ for $0<\mu<1$.

First observe that $b(v,w) = a(s,w)$ extends to $v=0$ in $C^\infty$ because $\psi$ does. In particular,
\begin{equation}
	b_v(v,w) = 2\psi_v(v,w)e^{2\psi(v,w)}
\end{equation}
also extends to $v=0$, and $b_v(0,w)$ does not depend on $w$. There is a constant $K$ such that $\vert b(v,w)\vert\leq Kv^\mu$, so $\vert a(s,w)\vert\leq Ke^{-s\mu} \leq K's^{-\mu}$, which gives $a \in s^{-\mu}C^0$.

Now $a_s(s,w) = 2\phi_s(s,w)e^{2\phi(s,w)}$, and we need to show that $\vert a_s(s,w)\vert \leq Ks^{-\mu}$. This follows from
\begin{equation}
	2\phi_s(s,w)e^{2\phi(s,w)} = -2e^{-s}\psi_v(v,w)e^{2\psi(v,w)}
\end{equation}
and the fact that $\psi_v(v,w)e^{2\psi(v,w)}$ extends to $v=0$ not depending on $w$.

Now the pattern appears for higher $y$-derivatives of $a$. When we express the derivatives in $v$ coordinates, the expressions gain a $v=e^{-s}$, so the decay follows since $\psi$ extends $C^\infty$ to $v=0$ without $w$ dependence. 

For the case of taking a $w$-derivative, its norm in $C^1$ is $e^{2s}a_w$ (the metric to define $C^{k,\alpha}$ is $ds^2 + e^{-2s}dw^2$). We have
\begin{equation}
	e^{2s}2\phi_w(s,w)e^{2\phi(s,w)} = v^{-2}\psi_w(v,w)e^{2\psi(v,w)},
\end{equation}
which is bounded by some power of $v$ because $\psi_w$ has infinite order $0$ at $v=0$. Similarly, we extend our argument to further $w$-derivatives. This shows that we can run Ricci flow with initial condition any metric considered by \cite{MoroianuGuillarmouRochon}.

Finally, to observe that the logarithm of the conformal factors along the Ricci flow still belong to $C^{\infty}_r(\M)$, note that the curvature $e^{-2\psi}(-1+\Delta\psi)$ is in $C^{\infty}_r(\M)$. This makes the flow to preserve the set of metrics with such conformal factors, which in turn allows the argument to extend to this generality.

\section{Corrected renormalized volume}\label{sec:corrected}

Let $M$ be a closed compact hyperbolic $3$-manifold with an oriented incompressible surface $\Sigma$ ($g>2$) which divides $M$ into two components $M_1, M_2$. We can consider now hyperbolic structures $N_1, N_2$ on the interiors of $M_1, M_2$ such that we have the inclusions $M_1\subset N_1, M_2\subset N_2$ by taking coverings with respect to $\pi_1(M_1),\pi_1(M_2)$, respectively. The fact that they glue along $\Sigma$ tells us that the conformal classes at infinity of the open manifolds are each others skinning maps.

Now, for $r$ sufficiently large, $M_1$ lives inside $N_1^r$, the $r$-leaf of the foliation defined by the renormalized volume, and hence
\begin{equation}
	\text{vol}(M_1) = V_R(M_1,c_1) + r\pi\chi(\Sigma) - \text{vol}(({\rm int} N_1^r)\setminus M_1) + \frac{1}{4} \int_{N_1^r} Hda,
\end{equation}
where $c_1$ is the conformal class at infinity for ${\rm int}(M_1)$.

We can take $r$ even larger so that we have a similar statement for $M_2$. Observe that the regions $({\rm int}N_1^r)\setminus M_1, ({\rm int}N_2^r)\setminus M_2$ glue along $\Sigma$ and live inside the quasi-Fuchsian manifold obtained by gluing $N_1\setminus M_1, N_2\setminus M_2$ along $\Sigma$. In particular the quasi-Fuchsian manifold $(N_1\setminus M_1)\cup(N_2\setminus M_2)$ has $(c_1,c_2)$ as a conformal class at infinity and $N_1^r, N_2^r$ as the $r$-leaves for the renormalized volume, from where we conclude
\begin{align}\begin{aligned}
	\text{vol}(({\rm int}N_1^r)\setminus M_1 \cup ({\rm int}N_2^r)\setminus M_2) &= V_R(\Sigma\times\left[ 0,1\right], c_1,c_2)  + \frac{1}{4} \int_{N_1^r} Hda
	\\ &+ \frac{1}{4} \int_{N_2^r} Hda + 2r\pi\chi(\Sigma).
	\end{aligned}
\end{align}
This gives the following proposition.

\begin{prop} For $M$ as above, ${\rm vol}(M) = V_R(M_1,c_1) + V_R(M_2,c_2) - V_R(\Sigma\times\left[ 0,1\right], c_1,c_2)$
\end{prop}

In order to use this proposition to find a lower bound for the volume of $M$ (and since $c_1, c_2$ are related by the skinning map), we want to show that the minimum of the corrected renormalized volume (defined as follows) is attained at the geodesic class.

\begin{defi}
Let $M$ be a compact acylindrical $3$-manifold with hyperbolic interior, connected boundary of genus $g>1$, and let $c\in\mathcal{T}(\partial M)$ the element that defines the conformal boundary at infinity. Then the corrected renormalized volume is defined as
\begin{equation}
	\overline{V}_R(M) = V_R(M,c) - \frac{1}{2}V_R(\partial M \times \left[0,1 \right], (c,\sigma(c)) ).
\end{equation}
\end{defi}

Observe that proposition (1) implies that, under cutting, the volume of $M$ is equal to the sum of the corrected renormalized volumes of its parts. It is straightforward to extend this to the case when $\Sigma$ has multiple components, and if $M$ is open without cusps, we can replace $\text{vol}(M)$ by $\overline{V}_R(M)$ in Proposition (1). Then if we consider the corrected renormalized volume to be an extension of the volume for closed hyperbolic manifolds, we obtain as a corollary of proposition 1:

\begin{cor} $\overline{V}_R$ is additive under cutting.
\end{cor}

Now, from the first variation of $V_R$ at $\I=\I^+$, we have
\begin{align}\begin{aligned}
	8 D\overline{V}_R(v) =  - \int_{\partial M^+}\langle D\mathrm{I}^+(v,d\sigma(v)),\II_0^+ (c,\sigma(c))\rangle da^+ 
	+ \int_{\partial M^-}\langle D\mathrm{I}^-(v,d\sigma(v)),\II_0^- (c,\sigma(c)) \rangle da^-,
	\end{aligned}
\end{align}
where $da^+, da^-$ are the volume forms for $\I^+, \I^-$, respectively.

Observe that if we take $c$ to be the geodesic class, $\II_0^+, \II_0^-$ are both zero, and hence $c$ is also a critical point for the corrected renormalized volume. Moreover, the Hessian is expressed as
\begin{align}
	8\text{Hess}\overline{V}_R (v,w) = - \langle D\mathrm{I}^+(v,d\sigma(v)), D\II_0^+ (w,\sigma(w))\rangle \rangle + \langle D\mathrm{I}^-(v,d\sigma(v)), D\II_0^- (w,\sigma(w)) \rangle,
\end{align}
which by Lemma \ref{II0var} is equal to
\begin{equation}
	8\text{Hess}\overline{V}_R (v,w) =  - \langle v, -\frac{1}{4}(w- d\sigma(w))\rangle   +  \langle d\sigma(v), \frac{1}{4}(w- d\sigma(w))\rangle. 
\end{equation}
Then
\begin{equation}
	32\text{Hess}\overline{V}_R (v,w) =  \langle  v+d\sigma(v) , w-d\sigma(w) \rangle,
\end{equation}
and since all eigenvalues of $d\sigma$ are between $-1$ and $1$, we obtain the following result.

\begin{theorem}
Let $M$ be a compact acylindrical $3$-manifold with hyperbolic interior, $\partial M \neq \emptyset$ without cusps, and $c\in\mathcal{T}(\partial M)$ be the geodesic class. Then $c$ is a local minimum for the corrected renormalized volume of $M$.
\end{theorem}

\bibliographystyle{amsalpha}
\bibliography{mybib}

\end{document}